\documentclass{amsart}
\usepackage[english]{babel}
\usepackage{amssymb}
\usepackage{hyperref}   
\usepackage{mathtools}
\mathtoolsset{showonlyrefs}

\theoremstyle{plain}
\newtheorem{definition}{Definition}
\theoremstyle{plain}
\newtheorem{lemma}{Lemma}
\theoremstyle{remark}
\newtheorem{remark}{Remark}
\theoremstyle{plain}
\newtheorem{theorem}{Theorem}

\begin{document}

\title{A note on Minkowski formula of conformal Killing-Yano 2-form}

\author{Xiaoxiang Chai}
\address{Korea Institute for Advanced Study, Seoul 02455, South Korea}
\email{xxchai@kias.re.kr}

\begin{abstract}
  We study the Minkowski formula of conformal Killing-Yano two-forms in a
  spacetime of constant curvature. We establish the spacetime Alexandrov
  theorem with a free boundary.
\end{abstract}

{\maketitle}

\section{Introduction}

The Minkowski formula states that for a smooth closed hypersurface $X : \Sigma
\to \mathbb{R}^n$,
\begin{equation}
  (n - k) \int_{\Sigma} \sigma_{k - 1} \mathrm{d} \mu = k \int_{\Sigma}
  \sigma_k \langle X, \nu \rangle .
\end{equation}
Here $\sigma_k$ is the $k$-th elementary symmetric functions of principal
curvatures of $\Sigma$. It has found itself many applications in Riemannian
geometry for example a proof of the celebrated Alexandrov theorem which says
that an closed embedded hypersurface of constant mean curvature must be an
sphere. The same ideas of proof lead to a free boundary generalization due to
Wang-Xia {\cite{wang-uniqueness-2019}} establishing the rigidity of spherical
caps in balls of space forms. Both closed and the free boundary settings made
use of a specially chosen conformal Killing vector field. Tachibana introduced
the conformal Killing-Yano two-form as a generalization of conformal Killing
vector field.

\begin{definition}[Tachibana {\cite{tachibana-conformal-1969}}]
  \label{cky}A two-form $Q$ on an $(n + 1)$-dimensional spacetime is called a
  conformal Killing-Yano 2-form if for every vector field $X, Y$ and $Z$ the
  following identity holds
  \begin{equation}
    (\nabla_X Q) (Y, Z) + (\nabla_Y Q) (X, Z) = [2 \langle X, Y \rangle
    \langle \xi, Z \rangle - \langle X, Z \rangle \langle \xi, Y \rangle -
    \langle Y, Z \rangle \langle \xi, X \rangle] \label{conformal killing
    yano}
  \end{equation}
  where $\xi = \frac{1}{n} \ensuremath{\operatorname{div}}Q$. We call $\xi$
  the associated 1-form of $Q$.
\end{definition}

In physics literature, these two forms are usually termed as {\itshape{hidden
symmetry}} and can give information about the spacetime. See for example
{\cite{jezierski-conformal-2006}} and the references therein. Besides its
physical significance, mathematically the conformal Killing-Yano two-forms are
also interesting. In particular, they also allow a Minkowski type formula.
Chen, Wang, Yau {\cite{chen-minkowski-2019}} expressed quasilocal masses using
this Minkowski formula. The authors of {\cite{wang-minkowski-2017}}
established a spacetime version of the Alexandrov theorem for codimension two
spacelike hypersurfaces via the Minkowski formula.

In this work, we are going to extend results in {\cite{wang-minkowski-2017}}
where they used only conformal Killing-Yano two-forms $r \mathrm{d} r \wedge
\mathrm{d} t$. First we state the spacetime CMC condition with free boundary.

\begin{definition}
  We say that $\Sigma^2$ in a spacetime $\mathbb{R}^{3, 1}$ is CMC with free
  boundary if $\Sigma$ admits a null normal vector field $L$ with $\langle
  \vec{H}, L \rangle$ is constant, $(D L)^{\bot} = 0$ and $\Sigma$ meets the
  de Sitter sphere $\mathbf{S}^{2, 1}$ orthogonally.
\end{definition}

Of course, one can allow arbitrary spacetime and boundary in the above
definition. One interesting problem related to such surfaces is the uniqueness
problem of a topological disk (cf. {\cite{fraser-uniqueness-2015}}). Without
the free boundary condition, similar questions can be asked for two-spheres in
$3 + 1$ dimensional de Sitter sphere (cf. {\cite{chern-proofs-1969}}) . One
can also ask whether a spacelike graph over $\mathbb{R}^2$ in $\mathbb{R}^{3,
1}$ with $\langle \vec{H}, L \rangle = 0$ and $(D L)^{\bot} = 0$ is linear
which is analogous to the Bernstein problem for minimal graphs.

We generalize the spacetime Alexandrov theorem to the free boundary settings
via establishing a spacetime Heintz-Karcher inequality Theorem \ref{spacetime
heintz-karcher}. We state here the theorem in the Minkowski spacetime.

\begin{theorem}
  \label{alexandrov}Let $\Sigma$ be a codimension two, future incoming null
  embedded submanifold in the $(3 + 1)$-dimensional Minkowski spacetime with
  free boundary on the de Sitter sphere $\mathbf{S}^{2, 1}$. If $\Sigma$ lies
  in a half spacetime, and there exists a null vector field $\underline{L}$
  such that along $\Sigma$ that $\langle \vec{H}, \underline{L} \rangle$ is a
  positive constant and $(D \underline{L})^{\bot} = 0$. Then $\Sigma$ lies in
  a shear free null hypersurface.
\end{theorem}

The theorem is a direct corrollary from Theorem \ref{spacetime heintz-karcher}
and similar proofs as in {\cite[Theorem 3.14]{wang-minkowski-2017}}. The
article is organized as follows:

In Section \ref{preliminaries}, we collect basics of spacetime of constant
curvature and the conformal Killing-Yano two-forms they admit. In Section
\ref{s:alexandrov}, we prove a spacetime Heintz-Karcher inequality with a free
boundary leading to a free boundary, spacetime Alexandrov theorem. We mention
briefly the generalization to higher order curvatures.

{\bfseries{Acknowledgements}} I would like to thank Xia Chao, Wang Ye-kai for
their interest and advice in an earlier version of this work. I would also
like to acknowledge the support of Korea Institute for Advanced Study under
the research number MG074401.

\section{Conformal Killing-Yano 2-form on spacetime of constant
curvature}\label{preliminaries}

A spacetime of dimension $3 + 1$ can only admit 20 conformal Killing-Yano
two-forms. Actually, if a spacetime admits all twenty of them, then the
spacetime has to be a spacetime of constant curvature. Note that for similar
statements are also true for conformal Killing vector fields. In Minkowski, de
Sitter and anti-de Sitter spacetime, these two forms are found explicitly. See
the works by Jezierski and Lukasik
{\cite{jezierski-conformal-2006,jezierski-conformal-2008}}. Now we collect
some basics of these spacetimes and the conformal Killing-Yano two forms that
live on them.

\subsection{Minkowski spacetime}

Let $(x^0, x^1, x^2, x^3)$ be the standard coordinates of the Minkowski space
$\mathbb{R}^{3, 1}$, define
\begin{align}
  \mathcal{D} & = - x^0 \mathrm{d} x^0 + x^1 \mathrm{d} x^1 + x^2 \mathrm{d}
  x^2 + x^3 \mathrm{d} x^3, \\
  \mathcal{T}_0 & = - \mathrm{d} x^0, \\
  \mathcal{T}_i & = \mathrm{d} x^i, \\
  \mathcal{L}_{0 i} & = - x^0 \mathrm{d} x^i + x^i \mathrm{d} x^0 . 
\end{align}
The conformal Killing-Yano 2-forms on Minkowski spacetime $\mathbb{R}^{3, 1}$
are
\begin{equation}
  \mathcal{T}_{\mu} \wedge \mathcal{T}_{\nu}, \mathcal{D} \wedge
  \mathcal{T}_{\mu}, \ast (\mathcal{D} \wedge \mathcal{T}_{\mu}) \text{ and }
  \mathcal{D} \wedge \mathcal{L}_{\mu \nu} + \frac{1}{2} \langle \mathcal{D},
  \mathcal{D} \rangle \mathcal{T}_{\mu} \wedge \mathcal{T}_{\nu},
  \label{minkowski killing 2-form}
\end{equation}
where $\ast$ is the Hodge star operator and $\mu, \nu$ range from 0 to 3. See
{\cite{jezierski-conformal-2006}} for a calculation. Note that all are still
conformal Killing-Yano 2-forms on $\mathbb{R}^{n, 1}$ except $\ast
(\mathcal{D} \wedge \mathcal{T}_{\mu})$.

We remark that the last one in {\eqref{minkowski killing 2-form}} can be used
to prove formulas relating the center of mass (See
{\cite{miao-evaluation-2016}}) and a Brown-York type quasi-local quantity by
following similar procedures in {\cite{chen-minkowski-2019}}.

\subsection{Anti-de Sitter spacetime}\label{anti de sitter basics}

We recall some basics of four-dimensional anti-de Sitter spacetime. The
anti-de Sitter spacetime ad$\mathbf{S}^{3, 1}$ is defined to be the set in
$\mathbb{R}^{3, 2}$
\begin{equation}
  - (y^0)^2 + (y^1)^2 + (y^2)^2 + (y^3)^2 - (y^4)^2 = - 1
\end{equation}
with metric induced from $\eta = - (\mathrm{d} y^0)^2 + (\mathrm{d} y^1)^2 +
(\mathrm{d} y^2)^2 + (\mathrm{d} y^3)^2 - (\mathrm{d} y^4)^2$. We will use
coordinates of the Poincar{\'e} ball model by setting $r = \sqrt{(x^1)^2 +
(x^2)^2 + (x^3)^2}$, $y^0 = \tfrac{1 + r^2}{1 - r^2} \cos t$, $y^4 = \tfrac{1
+ r^2}{1 - r^2} \sin t$ and $y^i = \tfrac{2 x^i}{1 - r^2}$. The metric of
ad$\mathbf{S}^{3, 1}$ is then $- (\tfrac{1 + r^2}{1 - r^2})^2 \mathrm{d} t^2 +
\tfrac{4 \sum_i (\mathrm{d} x^i)^2}{(1 - r^2)^2}$.

It is shown in {\cite{jezierski-conformal-2008}} that the conformal
Killing-Yano 2-forms in four-dimensional anti-de Sitter spacetime are
\begin{equation}
  \mathrm{d} y^0 \wedge \mathrm{d} y^i, \mathrm{d} y^0 \wedge \mathrm{d} y^4,
  \mathrm{d} y^i \wedge \mathrm{d} y^4, \mathrm{d} y^i \wedge \mathrm{d} y^j
\end{equation}
and their Hodge duals with respect to the anti-de Sitter metric. We fix the
frame $\theta^i = \tfrac{2}{1 - r^2} \mathrm{d} x^i$ and $\theta^0 = \tfrac{1
+ r^2}{1 - r^2} \mathrm{d} t$. Let $\omega = \tfrac{2}{1 - r^2} \mathrm{d} r$,
then the length of $\omega$ is one. We have
\begin{equation}
  \mathrm{d} y^i = \theta^i + y^i r \omega, \mathrm{d} y^4 = \cos t \theta^0 +
  \tfrac{2 r}{1 - r^2} \omega \sin t.
\end{equation}
Note that $y^4$ and $y^i$ are static potentials, that is $\nabla_i \mathrm{d}
y^{\mu} = y^{\mu} \theta^i$ and $\nabla_0 \mathrm{d} y^{\mu} = - y^{\mu}
\theta^0$ for each $\mu = 0, 1, \ldots, 4$. Here $\nabla_{\mu}$ denotes
covariant derivative with respect to the vector field
$(\theta^{\mu})^{\sharp}$. Then it is easy to obtain that
\begin{equation}
  \ensuremath{\operatorname{div}} (\mathrm{d} y^i \wedge \mathrm{d} y^4) = 3
  (y^i \mathrm{d} y^4 - y^4 \mathrm{d} y^i) .
\end{equation}
Note that $y^i \mathrm{d} y^4 - y^4 \mathrm{d} y^i$ is a Killing 1-form. Using
the properties of Hodge operators, we find that
$\ensuremath{\operatorname{div}} (\ast (\mathrm{d} y^2 \wedge \mathrm{d}
y^3))$ vanishes.

We remark that the 2-form $\mathrm{d} y^i \wedge \mathrm{d} y^4$ can be used
similarly as in {\cite{chen-conserved-2015}} to recover a formula relating the
integrals of Ricci tensor and Brown-York type mass vector of an asymptotically
hyperbolic manifold. These formulas are overlooked by the authors of
{\cite{chen-conserved-2015}}. The original proof is due to
{\cite{miao-quasi-local-2017}}.

\subsection{de Sitter spacetime}

The case with de Sitter spacetime is similar to the anti-de Sitter case (See \
{\cite{jezierski-conformal-2008}}). We consider here the $3 + 1$ dimensional
case i.e. $\mathbf{S}^{3, 1}$. The de Sitter spacetime is the subset
\begin{equation}
  y_0^2 + y_1^2 + y_2^2 + y_3^2 - y_4^2 = 1
\end{equation}
in $\mathbb{R}^{4, 1}$ with the metric inherited from the standard Lorentz
metric of $\mathbb{R}^{4, 1}$. We use the coordinate change
\begin{align}
  y^0 & = \tfrac{1 - r^2}{1 + r^2} \cosh t, \\
  y^i & = \tfrac{2 x^i}{1 + r^2} \\
  y^4 & = \tfrac{1 - r^2}{1 + r^2} \sinh t, 
\end{align}
where $r = \sqrt{\sum_{i = 1}^3 (x^i)^2} < 1$. Now the metric of the de Sitter
spacetime $\mathbf{S}^{3, 1}$ takes the form
\begin{equation}
  \eta = - (\tfrac{1 - r^2}{1 + r^2})^2 \mathrm{d} t^2 + \tfrac{4}{(1 +
  r^2)^2} [(\mathrm{d} x^1)^2 + (\mathrm{d} x^2)^2 + (\mathrm{d} x^3)^2] .
\end{equation}
It is shown in {\cite{jezierski-conformal-2008}} that the conformal
Killing-Yano 2-forms in four-dimensional de Sitter spacetime are
\begin{equation}
  \mathrm{d} y^0 \wedge \mathrm{d} y^i, \mathrm{d} y^0 \wedge \mathrm{d} y^4,
  \mathrm{d} y^i \wedge \mathrm{d} y^4, \mathrm{d} y^i \wedge \mathrm{d} y^j
\end{equation}
and their Hodge duals with respect to the de Sitter metric. We fix the frame
$\theta^i = \tfrac{2}{1 + r^2} \mathrm{d} x^i$ and $\theta^0 = \tfrac{1 -
r^2}{1 + r^2} \mathrm{d} t$. Let $\omega = \tfrac{2}{1 + r^2} \mathrm{d} r$,
then the length of $\omega$ is one. We have
\begin{equation}
  \mathrm{d} y^i = \theta^i - y^i r \omega, \mathrm{d} y^4 = \cosh t \theta^0
  - \tfrac{2 r}{1 + r^2} \omega \sinh t.
\end{equation}
Note that $y^4$ and $y^i$ are static potentials, that is $\nabla_i \mathrm{d}
y^{\mu} = - y^{\mu} \theta^i$ and $\nabla_0 \mathrm{d} y^{\mu} = y^{\mu}
\theta^0$ for each $\mu = 0, 1, \ldots, 4$. Here $\nabla_{\mu}$ denotes
covariant derivative with respect to the vector field
$(\theta^{\mu})^{\sharp}$. Then it is easy to obtain that
\begin{equation}
  \ensuremath{\operatorname{div}} (\mathrm{d} y^i \wedge \mathrm{d} y^4) = - 3
  (y^i \mathrm{d} y^4 - y^4 \mathrm{d} y^i) . \label{div <mathd>s}
\end{equation}
Note that $y^i \mathrm{d} y^4 - y^4 \mathrm{d} y^i$ is a Killing 1-form. We
found also easily that $\ensuremath{\operatorname{div}} (\ast (\mathrm{d} y^2
\wedge \mathrm{d} y^3))$ vanishes.

\section{Spacetime Alexandrov theorem with free boundary}\label{s:alexandrov}

We start by proving a Minkowski formula for a codimension two spacelike
hypersurface in $\mathbb{R}^{3, 1}$ with boundary meeting orthogonally with
the de Sitter sphere. The result is related to mean curvature only, the
generalization to higher order curvatures is quite straightforward.

The Minkowski spacetime is used as a prototype. First, we fix a conformal
Killing-Yano 2-form in Minkowski spacetime $\mathbb{R}^{3, 1}$
\begin{equation}
  Q =\mathcal{D} \wedge \mathcal{L}_{0 i} + \frac{1}{2} [1 + \langle
  \mathcal{D}, \mathcal{D} \rangle] e^0 \wedge e^i . \label{cyk minkowski}
\end{equation}
The associated 1-form is $\xi : = \tfrac{1}{n}
\ensuremath{\operatorname{div}}Q =\mathcal{L}_{0 i}$ since

\begin{lemma}
  \label{divergence of minkowski 2-form}The divergence of the 2-form $Q
  =\mathcal{D} \wedge \mathcal{L}_{0 i} + \tfrac{1}{2} \langle \mathcal{D},
  \mathcal{D} \rangle \mathrm{d} x^0 \wedge \mathrm{d} x^i$ is given by
  $\ensuremath{\operatorname{div}}Q = 3\mathcal{L}_{0 i}$.
\end{lemma}

Define the {\itshape{Minkowski unit ball}}
\begin{equation}
  \mathcal{B}^{3, 1} = \{x \in \mathbb{R}^{3, 1} : \langle x, x \rangle
  \leqslant 1\} .
\end{equation}
The boundary of $\mathcal{B}^{3, 1}$ is the de Sitter sphere $\mathbf{S}^{2,
1}$. It is easy to check that $\mathcal{D}^{\sharp} \lrcorner Q$ is zero along
$\partial \mathcal{B}^{3, 1}$, so $Q$ has no components normal to
$\mathbf{S}^{2, 1}$.

\begin{theorem}
  \label{minkowski formula}Let $\Sigma$ be an immersed oriented spacelike
  codimension two submanifolds of the Minkowski spacetime $\mathbb{R}^{3, 1}$,
  $\partial \Sigma$ lies in the de Sitter sphere $\mathbf{S}^{2, 1}$ and
  $\Sigma$ meets $\mathbf{S}^{2, 1}$ orthogonally. For any null vector field
  $\underline{L}$ of $\Sigma$, we have
  \begin{equation}
    \int_{\Sigma} [(n - 1) \langle \xi, \underline{L} \rangle + Q (\vec{H},
    \underline{L}) + Q (\partial_a, (D^a \underline{L})^{\bot})] \mathrm{d}
    \mu = 0.
  \end{equation}
\end{theorem}

\begin{proof}
  Define $\mathcal{Q}= Q (\partial_a, \underline{L}) \mathrm{d} u^a$ on
  $\Sigma$ and the proof is almost the same as Theorem 2.2 of
  {\cite{wang-minkowski-2017}}. We include their proof for convenience. Let
  $\underline{\chi} = \langle D_a \underline{L}, \partial_b \rangle$. Consider
  the 1-form $\mathcal{Q}= Q (\partial_a, \underline{L}) \mathrm{d} u^a$, we
  have
  \begin{align}
    \ensuremath{\operatorname{div}}\mathcal{Q} & = \nabla_a \mathcal{Q}^a - Q
    (\nabla_a \partial_a, \underline{L}) \\
    & = (D^a Q) (\partial_a, \underline{L}) + Q (\vec{H}, \underline{L}) + Q
    (\partial_a, D^a \underline{L}) \\
    & = (n - 1) \langle \xi, \underline{L} \rangle + Q (\vec{H},
    \underline{L}) + \underline{\chi}_{a b} Q^{a b} + Q (\partial_a, (D^a
    \underline{L})^{\bot}) \\
    & = (n - 1) \langle \xi, \underline{L} \rangle + Q (\vec{H},
    \underline{L}) + Q (\partial_a, (D^a \underline{L})^{\bot}) . 
\end{align}
  Integration by parts and noting that $Q$ has no components normal to the de
  Sitter sphere.
\end{proof}

\subsection{A monotonicity formula}

Let $\Sigma$ be a spacelike submanifold of codimension two in a spacetime
$(\mathcal{S}^{3, 1}, g)$ which admits a Killing-Yano two form $Q$. Here,
$\mathcal{S}$ is either one of the four dimensional Minkowski, de Sitter and
anti de Sitter spacetime. We require that $Q$ has no normal component normal
to a support hypersurface $S$. Suppose that $\langle \vec{H}, \underline{L}
\rangle \neq 0$, we define the following functional
\begin{equation}
  \mathcal{F} (\Sigma, [\underline{L}]) = (n - 1) \int_{\Sigma} \frac{\langle
  \xi, \underline{L} \rangle}{\langle \vec{H}, \underline{L} \rangle}
  \mathrm{d} \mu - \frac{1}{2} \int_{\Sigma} Q (L, \underline{L}) \mathrm{d}
  \mu . \label{functional}
\end{equation}
Note $\mathcal{F}$ is invariant under the change $L \to a L$ and
$\underline{L} \to \frac{1}{a} \underline{L}$.

Let $\chi$ and $\underline{\chi}$ be respectively the second fundamental form
with respect to $L$, $\underline{L}$; let $\underline{C}_0$ denote the future
incoming null hypersurface of $\Sigma$.

$\underline{C}_0$ is obtained by taking the collection of all null geodesics
emanating from $\Sigma$ with initial velocity $\underline{L}$. We then extend
it to a future directed null vector field along $\underline{C}_0$. Consider
the evolution of $\Sigma$ along $\underline{C}_0$ by a family of immersions $F
: \Sigma \times [0, T) \to \underline{C}_0$ satisfying
\begin{equation}
  \left\{\begin{array}{l}
    \frac{\partial F}{\partial s} (x, s) = \varphi (x, s) \underline{L},\\
    F (x, 0) = F_0 (x),\\
    \Sigma \bot S
  \end{array}\right.
\end{equation}
for some positive function $\varphi (x, s)$.

We have the following monotonicity property of the flow $\varphi$.

\begin{theorem}
  Suppose that $\langle \vec{H}, \underline{L} \rangle > 0$ for some null
  vector field $\underline{L}$. Then $\mathcal{F} (F (\Sigma, s),
  [\underline{L}])$ is monotone decreasing under the flow.
\end{theorem}

\begin{proof}
  See Theorem 3.2 of {\cite{wang-minkowski-2017}}. We only have to use the
  extra fact that $Q$ has no components normal to the de Sitter space as in
  the proof of Theorem \ref{minkowski formula}.
\end{proof}

The monotonicity property leads to a spacetime Heintz-Karcher inequality.
More, specifically, if under certain flow $\varphi$, the surface $\Sigma$ with
$\langle \vec{H}, \underline{L} \rangle > 0$ flows into a submanifold of the
time slice $\{x^0 = 0\}$ at $s = T$ and for $\Sigma$
\begin{equation}
  \mathcal{F} (\Sigma, [L]) \geqslant 0
\end{equation}
holds provided $\varphi (\Sigma, T) \subset \{x^0 = 0\}$ and $\mathcal{F}
(\varphi (\Sigma, T), [L]) \geqslant 0$.

\begin{lemma}
  For any $\Sigma \subset \{x^0 = 0\}$, $\mathcal{F} (\Sigma, [L]) \geqslant
  0$ reduces to
  \begin{equation}
    (n - 1) \int_{\Sigma} \tfrac{x^i}{H} \mathrm{d} \mu \geqslant
    \int_{\Sigma} \langle X_{\partial_i}, \nu \rangle \mathrm{d} \mu .
    \label{space heintz-karcher}
  \end{equation}
  
\end{lemma}

\begin{proof}
  We have $\underline{L} = \partial_t - e_n$ where $e_n$ is a unit normal. So
  $\langle \vec{H}, \underline{L} \rangle = H$ where $H$ is the mean curvature
  of $\Sigma$ in $\mathbf{B}^n$. We have that $\xi =\mathcal{L}_{0 i} = x^i
  \mathrm{d} x^0$, so $\langle \xi, \underline{L} \rangle = x^i$. Also,
  \begin{equation}
    Q (L, \underline{L}) = Q (\partial_t + \nu, \partial_t - \nu) = 2 \langle
    X_{\partial_i}, \nu \rangle,
  \end{equation}
  where $X_a = \langle X, a \rangle X + \tfrac{1}{2} (|X|^2 + 1) a$ where $a =
  a^i \partial_i$ is a constant vector in $\mathbb{R}^n$. It easily leads to
  {\eqref{space heintz-karcher}}.
\end{proof}

Note that this is precisely an inequality proven already by Wang-Xia
{\cite[(5.5)]{wang-uniqueness-2019}} with the assumption that $\Sigma$ has
positive mean curvature and lies in a half ball.

Combining with their result, we have

\begin{theorem}[spacetime Heintz-Karcher inequality]
  \label{spacetime heintz-karcher} If there exists a flow $\varphi$ of a
  hypersurface $\Sigma$ with $\langle \vec{H}, \underline{L} \rangle > 0$ for
  some null vector field $\underline{L}$ and a free boundary on
  $\mathbf{S}^{2, 1}$ which flows $\Sigma$ into the half unit ball of the
  slice $\{x^0 = 0\}$, then we have the inequality
  \begin{equation}
    \int_{\Sigma} \frac{\langle \xi, \underline{L} \rangle}{\langle \vec{H},
    \underline{L} \rangle} \mathrm{d} \mu \geqslant \frac{1}{2 (n - 1)}
    \int_{\Sigma} Q (L, \underline{L}) \mathrm{d} \mu .
  \end{equation}
  Equality occurs if and only if $\Sigma$ lies in a shear free null
  hypersurface with free boundary on $\mathbf{S}^{n - 1, 1}$.
\end{theorem}

\begin{proof}
  Let $\Sigma_t = \varphi_t (\Sigma)$, then for each $t > 0$, the equality
  holds. Suppose that $\Sigma_T \subset \{x^0 = 0\}$ for some $T > 0$. So
  $\Sigma_T$ has to be a spherical cap orthogonal to the unit sphere in
  $\mathbb{R}^n$ according to {\cite{wang-uniqueness-2019}}. In particular,
  under the flow $\varphi$, $\Sigma_t$ foliates a shear free null hypersurface
  $S$ with free boundary.
\end{proof}

\subsection{Anti-de Sitter case}\label{anti de sitter case}

Theorems \ref{minkowski formula}, \ref{spacetime heintz-karcher} and
\ref{alexandrov} work as well in the case with $\partial \Sigma = \emptyset$.
The same proof also adapts in the anti-de Sitter and de Sitter settings. We
use the notations in Section \ref{anti de sitter basics}. For simplicity, we
set $i$ to be 1, we use the 2-forms $\mathrm{d} y^1 \wedge \mathrm{d} y^4$ and
$\ast (\mathrm{d} y^2 \wedge \mathrm{d} y^3)$ only. Note that the Hodge star
operator commutes with the covariant derivative. Using this, we see easily
that $\ensuremath{\operatorname{div}} (\ast (\mathrm{d} y^2 \wedge \mathrm{d}
y^3))$ vanishes. We use the 2-form

\begin{flushleft}
  \[ Q = \mathrm{d} y^1 \wedge \mathrm{d} y^4 + l \ast (\mathrm{d} y^2 \wedge
     \mathrm{d} y^3) \label{a<mathd>s killing-yano 2-form} \]
\end{flushleft}

where $l > 0$ is a positive constant. We define the surface $\mathcal{B}^{3,
1}$ to be the surface with distance less than $d$ from the point $t = 0$, $r =
0$ where $\cosh d = l$. If $Y_1, Y_2 \in \ensuremath{\operatorname{ad}}
\mathbf{S}^{3, 1}$ (using the embedding into $\mathbb{R}^{3, 2}$) are two
points which can be connected via a spacelike geodesic, then the distance from
$Y_1$ to $Y_2$ is $\cosh d = - \eta (Y_1, Y_2) > 0$. The boundary $S =
\partial \mathcal{B}^{3, 1}$ is a timelike hypersurface of dimension three of
constant distance from the point $t = 0$, $r = 0$ and it is umbilical hence
null geodesics intrinsic to $S$ are also null geodesic in ad$\mathbf{S}^{3,
1}$. It is the analog of de Sitter sphere which is of constant distance to the
origin in Minkowski spacetime. It is a tedious task to check that $Q$ has no
component normal to $S$. We state here the spacetime Heintz-Karcher inequality
and leave the spacetime Alexandrov theorem to the reader.

\begin{theorem}
  \label{spacetime heintz-karcher in ads}(spacetime Heintz-Karcher inequality
  in $\mathcal{B}^{3, 1}$) If there exists a flow $\varphi$ of a hypersurface
  $\Sigma$ with $\langle \vec{H}, \underline{L} \rangle > 0$ for some null
  vector field $\underline{L}$ and a free boundary on $S$ which flows $\Sigma$
  into the half geodesic ball of the slice $\{t = 0\}$, then we have the
  inequality
  \begin{equation}
    \int_{\Sigma} \frac{\langle \xi, \underline{L} \rangle}{\langle \vec{H},
    \underline{L} \rangle} \mathrm{d} \mu \geqslant \frac{n}{2 (n - 1)}
    \int_{\Sigma} Q (L, \underline{L}) \mathrm{d} \mu . \label{spacetime
    heintz karcher formula in ads}
  \end{equation}
  Equality occurs if and only if $\Sigma$ lies in a shear free null
  hypersurface.
\end{theorem}

\begin{proof}
  The proof is the same with Theorem \ref{spacetime heintz-karcher}. We only
  have to verify when $t = 0$ the inequality holds. Let $\nu$ be the unit
  normal of $\Sigma$ in the $\{t = 0\}$ slice. Indeed, when $t = 0$, $\xi =
  y^1 \mathrm{d} y^4$ and
  \[ \underline{L} = e_0 - \nu = \tfrac{1 - r^2}{1 + r^2} \partial_t - \nu, \]
  so
  \[ \langle \xi, \underline{L} \rangle = y^1 = \tfrac{2 x^1}{1 - r^2} . \]
  We turn to $Q (L, \underline{L})$. We have
  \[ (\mathrm{d} y^1 \wedge \mathrm{d} y^4) (L, \underline{L}) = 2 \mathrm{d}
     y^1 (\nu) \]
  and
  \[ (\mathrm{d} y^1)^{\sharp} = \tfrac{1}{2} \partial_1 + (x^1 x^j \partial_j
     - \tfrac{1}{2} r^2 \partial_1) . \]
  And
  \[ \ast (\mathrm{d} y^2 \wedge \mathrm{d} y^3) = - \theta^1 \wedge \theta^0
     + y^2 (x^1 \theta^2 - x^2 \theta^1) \wedge \theta^0 + y^3 (x^1 \theta^3 -
     x^3 \theta^1) \wedge \theta^0, \label{ads correction term} \]
  so the 1-form $(\ast (\mathrm{d} y^2 \wedge \mathrm{d} y^3)) (\cdot, e_0)$
  is dual to $- \tfrac{1}{2} \partial_1 + (x^1 x^j \partial_j - \tfrac{1}{2}
  r^2 \partial_1)$. As usual, $Q (L, \underline{L}) = 2 Q (\nu, e_0)$. Thus,
  \begin{equation}
    Q (L, \underline{L}) = 2 \langle X_{\partial_1}, \nu \rangle,
  \end{equation}
  where $X_a = (1 + l) \left[ x^k a_k x^j \partial_j - \tfrac{1}{2} (r^2 +
  \tfrac{l - 1}{l + 1}) a \right]$ with $a = a^j \partial_j$ being a constant
  vector in $\mathbb{R}^n$. Letting $l = \tfrac{1 + R_{\mathbb{R}}^2}{1 -
  R_{\mathbb{R}}^2}$, {\eqref{spacetime heintz karcher formula in ads}}
  reduces to also {\cite{wang-uniqueness-2019}}.
\end{proof}

\begin{remark}
  It is easy to check that the higher dimensional analog of $\ast (\mathrm{d}
  y^2 \wedge \mathrm{d} y^3)$ in the $n$-dimensional anti-de Sitter spacetime
  \[ \ensuremath{\operatorname{ad}} \mathbf{S}^n = \{- (y^0)^2 + (y^1)^2 +
     \cdots + (y^n)^2 - (y^{n + 1})^2 = 1\} \]
  is
  \begin{equation}
    - e^1 \wedge e^0 + \sum_{i \neq 1}^n y^i (x^1 e^i - x^i e^1) \wedge e^0 .
  \end{equation}
\end{remark}

\subsection{de Sitter case}

We calculate below the quantities needed for a theorem parallel to Theorem
\ref{spacetime heintz karcher formula in ads}. We follow similar notations and
omit the the statements or details. Generalizing to higher dimension is also
straightforward. The conformal Killing-Yano 2-form is
\[ Q = \mathrm{d} y^4 \wedge \mathrm{d} y^1 + l \ast (\mathrm{d} y^3 \wedge
   \mathrm{d} y^2) \]
and its associated 1-form
\[ \xi =\ensuremath{\operatorname{div}}Q = 3 y^1 \mathrm{d} y^4 - y^4
   \mathrm{d} y^1 . \]
Notice the order of the superscripts. Within the slice $\{t = 0\}$, we have
that $\xi = y^1 \mathrm{d} y^4$ and $\underline{L} = e_0 - \nu = \tfrac{1 +
r^2}{1 - r^2} \partial_t - \nu$ and
\[ \langle \xi, \underline{L} \rangle = \tfrac{2 x^1}{1 + r^2} . \]
We turn to $Q (L, \underline{L})$. We have
\[ (\mathrm{d} y^4 \wedge \mathrm{d} y^1) (L, \underline{L}) = - 2 \mathrm{d}
   y^1 (\nu) . \]
Note that
\[ A := - (\mathrm{d} y^1)^{\sharp} = - \tfrac{1}{2} \partial_1 -
   (\tfrac{1}{2} r^2 \partial_1 - x^1 x^j \partial_j) \]
and
\[ \ast (\mathrm{d} y^3 \wedge \mathrm{d} y^2) = \theta^1 \wedge \theta^0 +
   y^2 (x^1 \theta^2 - x^2 \theta^1) \wedge \theta^0 + y^3 (x^1 \theta^3 - x^3
   \theta^1) \wedge \theta^0, \label{ds correction term} \]
so the 1-form $\ast (\mathrm{d} y^3 \wedge \mathrm{d} y^2) (\cdot, e_0)$ is
dual to
\[ B := \tfrac{1}{2} \partial_1 - \tfrac{1}{2} r^2 \partial_1 + x^1 x^j
   \partial_j . \]
$A + l B$ is then
\begin{equation}
  X_{\partial_1} : = (1 + l) \left[ x^1 x^j \partial_j + \tfrac{1}{2}
  (\tfrac{1 - l}{l + 1} - r^2) \partial_1 \right] .
\end{equation}
Therefore $Q (L, \underline{L}) = 2 \langle X_{\partial_1}, \nu \rangle$.
Setting $\tfrac{1 - l}{1 + l} = |x|^2$ with $0 < l < 1$ recovers the form of
{\cite{wang-uniqueness-2019}}. We have not given the support hypersurface of
the boundary yet. To this end, we fix a point $O = \{t = 0, r = 0\}$, let $S$
be the hypersurface in $\mathbf{S}^{3, 1}$ be a hypersurface of constant
distance $d$ from the point $O$ where $\cos d = l$. It is fairly easy to check
that $Q$ has no components to the hypersurface $S$.

\

\

\

\

\end{document}